\documentclass[10pt,a4paper]{scrartcl}

\usepackage[english]{babel}

\usepackage[section]{placeins}

\usepackage{url}

\usepackage{xcolor}

\usepackage{a4wide}
\usepackage{hyperref}
\usepackage[numbers]{natbib}

\usepackage{amsmath,amssymb,mathrsfs,bbm,cases}
\usepackage{mathtools}
\usepackage{amsthm}

\usepackage{graphicx}
\usepackage{subcaption}

\usepackage{enumerate}

\allowdisplaybreaks

\theoremstyle{plain}
\newtheorem{thm}{Theorem}[section]

\newtheorem{lem}[thm]{Lemma}

\theoremstyle{definition}

\newtheorem{alg}[thm]{Algorithm}

\theoremstyle{remark}

\newcommand{\R}{\mathbbm{R}}

\newcommand{\N}{\mathbbm{N}}

\renewcommand{\P}{\mathbf{P}}

\newcommand{\E}{\mathbf{E}}

\newcommand{\V}{\mathrm{Var}}

\newcommand{\Cov}{\mathrm{Cov}}

\newcommand{\e}{\mathrm{e}}

\newcommand{\vt}{\vartheta}

\newcommand\indi{\protect\mathpalette{\protect\indiT}{\perp}}
\def\indiT#1#2{\mathrel{\rlap{$#1#2$}\mkern2mu{#1#2}}}

\begin{document}

\title{On generating fully discrete samples of the stochastic heat equation on an interval}
\author{Florian Hildebrandt}
\date{Universit\"at Hamburg}
\maketitle

\begin{abstract}
Generalizing an idea of \citet{DavieGaines2001}, we present a method for the simulation of fully discrete samples of the solution to the stochastic heat equation on an interval. We provide a condition for the validity of the approximation, which holds particularly when the number of temporal and spatial observations tends to infinity. Hereby, the quality of the approximation is measured in total variation distance. In a simulation study we calculate temporal and spatial quadratic variations from sample paths generated both  via our method and via naive truncation of the Fourier series representation of the process. 
Hereby, the results provided by our method are more accurate at a considerably lower computational cost.  
\end{abstract}

\noindent\textbf{Keywords:} Stochastic heat equation, simulation, total variation distance, high frequency observations, power variations
\smallskip

\noindent\textbf{2010 MSC:} 60H15, 65C30, 62E17   

\section{Introduction}
In this article we consider a method for generating discrete samples $X_{t_i}(y_k)$ on a regular grid $((t_i,y_k),\,0\leq i\leq N,0\leq k\leq M)\subset [0,T]\times[0,1]$, where $X$ is the weak solution to the stochastic partial differential equation (SPDE)
\begin{equation} \label{eq:SPDEintro}
\begin{cases}
dX_t(x) = \left(\vt_2\frac{\partial^2}{\partial x ^2} X_t(x) + \vt_1\frac{\partial}{\partial x } X_t(x) +\vt_0 X_t(x)\right)\,dt+\sigma\,dW_t(x),\quad x\in[0,1],t\in [0,T],\\
X_t(0)=X_t(1)=0, \\
X_0 = \xi.
\end{cases}
\end{equation}
Hereby, $dW$ denotes space-time white noise, $\xi$ is some initial condition independent of $dW$ and $T\in (0,\infty]$.\\

The process defined by \eqref{eq:SPDEintro} has recently gained considerable interest in the area of mathematical statistics, the focus being the  problem of estimating the parameters $(\sigma^2,\vt_2,\vt_1,\vt_0)$ based on discrete space-time observations, see \cite{Cialenco17,Bibinger18,Uchida19,HildebrandtTrabs2019}. As the primary foundation for their analysis and simulations authors have used the fact that the solution of \eqref{eq:SPDEintro} admits a representation $X_t(y)=\sum _{\ell\geq 1}u_\ell(t)e_\ell(y)$, where $(u_\ell)_{\ell\geq 1}$ are independent one dimensional Ornstein-Uhlenbeck processes and $(e_\ell)_{\ell\geq 1}$ are the eigenfunctions of the differential operator associated with \eqref{eq:SPDEintro}. In particular, in order to simulate $X$ on a space-time grid, the approximation  $X_{t_i}^{\mathrm{trunc}}(y_k) = \sum _{\ell= 1}^K u_\ell(t_i)e_\ell(y_k)$ for some large integer $K$ appears natural in view of the increasing drift towards 0 of the processes $u_\ell$ for $\ell \to \infty$. Hereby, the processes $u_\ell$ can be simulated exactly based on their AR(1)-structure or via an exponential Euler scheme, see \cite{Cialenco17}. As empirically observed, e.g.~by \citet{Uchida19}, the value of $K$ has to be chosen carefully depending on the numbers of temporal and spatial observations $N$ and $M$. In fact, even for moderate sample sizes, large values of $K$ turn out to be crucial in order to prevent a severe bias in the simulated data. This makes  simulations very costly.\\

Generalizing an idea stated in \cite{DavieGaines2001}, in this article we analyze an alternative  approach, leading to almost exact (in distribution) discrete samples of $X$ at a considerably lower computational cost. The two key observations leading to the method are: firstly, the first $M$ rescaled eigenfunctions $e_\ell$ are orthorgonal with respect to the empirical inner product, which yields a representation of the spatially discrete data in terms of a finite number of eigenfunctions. Secondly, for large values of $\ell$ the process $(u_\ell(t_i),\,0\leq i\leq N)$ can be approximated well by a set of independent random variables. Here, the coefficient processes corresponding to high Fourier modes are replaced by a set of independent random variables rather than truncated, hence, we shall call this approach the \emph{replacement method}, as opposed to the \emph{truncation method}.\\
Denoting $\Delta=t_{i+1}-t_i$, our precise analysis reveals that it is sufficient to generate discrete samples of $J\geq M$ Ornstein-Uhlenbeck processes accompanied by a sample of the same size of independent normal random variables, as long as (roughly) $J/\sqrt \Delta \to \infty $. Hereby, the quality of the approximation is measured in terms of the total variation distance of the random vector $(X_{t_i}(y_k),0\leq i\leq N,0\leq k\leq M)$ from its approximation. Although the magnitude of the total variation distance in our convergence result is explicit, it is not informative in the sense of a rate of convergence since the reference measure changes with the values of $N$ and $M$.\\

The literature on approximation of SPDEs usually focuses on controlling errors of the type $\E(\Vert X(T)-X^{\mathrm{a}}(T)\Vert_{L^2})$ (strong sense) or  $\vert\E(\phi( X(T)))-\E(\phi(X^{\mathrm{a}}(T))\vert$ (weak sense) for an approximation $X^a$ of $X$, a fixed time instance $T$ and a continuous functional $\phi$, see e.g.~\cite{JentzenKloeden2009}. Our primary goal, on the other hand, is to mimic the distribution of the discrete observations $(X_{t_i}(y_k),\,0\leq i \leq N,\,0\leq k\leq M)$ as well as possible, particularly when at least one of the numbers $M$ and $N$ tends to infinity. This is an important task, for instance, with regard to computation of the asymptotic value of power variations, which are used  in the statistical theory for SPDEs, for example. The corresponding functionals, mapping sample paths to the asymptotic value of their power variations, are not continuous:  a function close to zero can have arbitrarily rough paths. Hence, the known bounds on the strong or weak approximation error do not provide conditions under which the approximate power variation is close to the true one, in general. Here,
controlling the total variation distance between the discrete sample and its approximation is an appropriate tool: given that the total variation distance becomes negligible, functionals computed from the approximation converge to the correct weak limit (if existent), see also the discussion following Theorem \ref{thm:TV}. 
We remark that \citet{ChongWalsh2012} examined the related question how finite difference approximations affect the asymptotic value of power variations of the stochastic heat equation.
\\

This article is organized as follows: in Section \ref{sec:Model} we give a precise definition of the probabilistic model and recall some of its properties. In Section \ref{sec:MainResult} we present the replacement method and state our convergence result. Section \ref{sec:Simulations} is devoted to a numerical example, particularly comparing our simulation method with the truncation method.  Finally, Section \ref{sec:Proofs} contains the proofs of our results.

\section{Probabilistic model} \label{sec:Model}
We consider the linear parabolic SPDE
\eqref{eq:SPDEintro} driven by a cylindrical Brownian motion $W$ where  $\xi \in L^2([0,1])$ is some initial condition independent of $W$.
More  precisely, we consider the weak solution $X = (X_t(x), \,t\geq 0 ,\, x \in [0,1])$ to $dX_t=A_\vt X_t\,dt+\sigma dW_t$
associated with the differential operator $A_\vt=\vt_2\frac{\partial^2}{\partial x ^2}+ \vt_1\frac{\partial}{\partial x }  +\vt_0 $. As usual, the Dirichlet boundary condition in \eqref{eq:SPDEintro} is implemented in the domain 
$\mathcal D (A_\vt) = H^2((0,1))\cap H_0^1((0,1))$ of $A_\vt$ where $H^k((0,1))$  denotes the $L^2$-Sobolev spaces of order $k\in\N$ and with $H_0^1((0,1))$ being the closure of $C_c^\infty((0,1))$ in $H^1((0,1))$. The cylindrical Brownian motion $W$ is defined as a linear mapping $L^2((0,1)) \ni u \mapsto W_\cdot(u)$ such that 
 $t\mapsto  W_t(u)$ is a one-dimensional standard Brownian motion for all normalized $u\in L^2([0,1])$
and such that the covariance structure is 
$\Cov\left( W_t(u) , W_s(v) \right)=(s\wedge t)\, \langle u,v \rangle,$ for $u,v\in L^2([0,1]),\,s,t\geq 0$. $W$ can thus be understood as the anti-derivative in time of space-time white noise.\\

Throughout, we assume that the parameters in \eqref{eq:SPDEintro} belong to the set
$$\Theta = \left\{(\sigma^2,\vt_2,\vt_1,\vt_0)\in \R^4:\,\sigma^2, \vt_2,\frac{\vartheta_1^2}{4\vartheta_2 ^2}-\frac{\vartheta_0}{\vartheta_2}+\pi^2>0\right\},$$ 
from which it follows that $A_\vt$ is a negative self-adjoint operator. Consequently, there is a unique weak solution of \eqref{eq:SPDEintro}, which is given by the variation of constants formula 
$
X_t  =\e^{tA_\vartheta}\xi + \sigma\int_0^t \e^{(t-s)A_\vartheta}\,dW_s,\, t\geq 0,
$ 
where $(\e^{tA_\vt})_{t\geq 0}$ denotes the strongly continuous semigroup generated by $A_\vt$, see \cite[Theorem 5.4]{DaPrato14}.

In order to derive a Fourier representation of $X$, consider $L^2[0,1]$ equipped with the weighted inner product 
\begin{align*}
\langle u, v \rangle:=\langle u, v \rangle_\vt:= \int_0^1 u(x)v(x)\e^{\kappa x}\,dx,\qquad
\text{where } \kappa = \frac{\vt_1}{\vt_2},
\end{align*}
$u,v \in L^2([0,1])$, such that $A_\vt$ admits a complete orthonormal system of eigenfunctions $(e_\ell)_{\ell \geq 1}$ with respective eigenvalues  $(\lambda_\ell)_{\ell \geq 1}$, namely
\begin{gather*} e_\ell (y) = \sqrt{2} \sin(\pi \ell y )\e^{-\kappa y/2},\quad
\lambda_\ell = \pi^2 \vt_2 \ell^2 +\frac{\vt_1^2}{4\vt_2}-\vt_0, \qquad  y\in [0,1],\,\ell\in \N.
\end{gather*}
The cylindrical Brownian motion can be realized via $W_t(\cdot)= \sum_{\ell\geq 1} \beta_\ell(t)\langle \cdot,e_k\rangle$ for a sequence of independent standard Brownian motions $(\beta_\ell)_{\ell \geq 1}$. Hence, in terms of the projections $u_\ell(t):=\langle X_t,e_\ell\rangle,\,t\ge0,\ell \in \N$, we obtain the representation
\begin{equation} \label{eq:SPDE_solution}
{X_t (x)}  {=\sum _{\ell\geq 1} u_\ell(t)e_\ell(x)}, \quad t \geq 0,\,x\in [0,1]. 
\end{equation}
Hereby, the coefficients $(u_\ell)_{\ell \ge1}$ are one dimensional independent Ornstein-Uhlenbeck processes, satisfying
$
du_\ell(t)= -\lambda_\ell u_\ell(t)\,dt +\sigma\, d\beta_\ell(t)
$
or, equivalently,
$$u_\ell(t)= u_\ell(0)\e^{-\lambda_\ell t}+\sigma\int_0^t \e^{-\lambda_\ell (t-s)}\,d\beta_\ell(s),\qquad u_\ell(0)=\langle\xi,e_\ell \rangle,$$
in the sense of a finite dimensional stochastic integral.

When $X$ is only considered at the discrete points $y_k= \frac{k}{M},\,k=0,\ldots,M,$ in space it is possible to further simplify the series representation \eqref{eq:SPDE_solution}. 
To that aim, we introduce the weighted empirical inner product $$\langle u,v \rangle_{M}:=\langle u,v \rangle_{\vt,M}:=\frac{1}{M}\sum_{m=0}^M u(y_k)v(y_k)\e^{\kappa y_k}$$ for functions $u,v:[0,1]\to \R$. Elementary trigonometric identities show that the first coefficient processes $(e_\ell,\,\ell\leq M-1)$ form an orhonormal basis with respect to $\langle\cdot,\cdot\rangle_M$, i.e.
$$\langle  e_\eta  , e_\nu  \rangle_M=\frac{2}{M}\sum_{k=0}^M \sin(\pi\eta y_k)\sin(\pi\nu y_k) = \delta_{\eta\nu},\qquad 1\leq \eta,\nu\leq M-1.$$
Therefore, in combination with the properties $\bar e_M =0,\,\bar e_{\eta+2\ell M}=\bar e_{\eta}$ and   $\bar e_{2M-\eta+2\ell M}=-\bar e_{\eta}$ for $\bar e_\ell :=(e_\ell(y_0),\ldots,e_\ell(y_M))\in \R^{M+1}$ and any  $\ell \geq 1$, we can pass to the representation
\begin{equation} \label{eq:X.yk}
X_{t}(y_k)= \sum_{m=1}^{M-1} U_m(t)e_m (y_k),\qquad t\geq 0,\, k=0,\ldots,M,
\end{equation}
$$ \text{where}\quad U_m(t)= \langle X_t(\cdot), e_m\rangle_M =\sum_{\ell \in \mathcal I_m^+} u_{\ell}(t)-\sum_{\ell \in \mathcal I_m^-} u_{\ell}(t)$$
with  $\mathcal{I}_m^+ = \{m+2\ell M,\,\ell \in \N_0\}$ and $\mathcal{I}_m^- =  \mathcal \{2M-m+2\ell M,\,\ell \in \N_0\}$.  Thus, for discrete observations on a grid, there is a representation of $X$ in terms of a finite number of  independent coefficient processes.\\

Regarding the initial condition $X_0=\xi$, we will focus on the two most important scenarios: One case, naturally playing an outstanding role, is that of a stationary initial distribution, where $u_\ell(0)=\langle\xi,e_\ell \rangle_\vt,\,\ell\geq 1,$ are independent with $u_\ell(0)\sim \mathcal N(0,\frac{\sigma^2}{2\lambda_\ell})$. The second one is a vanishing initial condition $X_0=0$. The particular importance of this case comes from the fact that the solution $X$ with an arbitrary initial condition $X_0=\xi$ can always be decomposed into $X_t= X_t^0+\e^{A_\vt t}X_0$, where $X^0$ is the solution with zero initial condition and $\e^{A_\vt t}X_0=\sum_{\ell \geq 1}\e^{-\lambda_\ell t}\langle X_0,e_\ell \rangle_\vt e_\ell$. In the sequel, we will use the notation $X^{\mathrm{st}},\,u_\ell^{\mathrm{st}}$ and $U_m^{\mathrm{st}}$ for the stationary solution and $X^0,\,u_\ell^{0}$ and $U_m^{0}$ for the solution starting in zero.\\ 

We end this section by introducing some notation: $\mathrm{TV}(P,Q)= \sup_{A\in \mathcal F}|P(A)-Q(A)|$ denotes the total variation distance between two probability measures $P$ and $Q$ on a common measurable space $(\Omega,\mathcal F)$. We also write $\mathrm{TV}(X,Y)$ for the total variation distance between the laws of two random variables $X$ and $Y$ with the same sample space. Further, for sequences $(a_n)$ and $(b_n)$ we write $a_n\lesssim b_n$ if there exists $C>0$ such that $|a_n| \leq C|b_n|$ for all $n\in \N$. The expression $a_n\eqsim b_n$ means that $a_n\lesssim b_n \lesssim a_n$. The Frobenius norm for matrices is denoted by $\Vert\cdot \Vert_F$ and, finally, the notation $M,N\to \infty$ is used in the sense of $\min(M,N)\to \infty$.

\section{Simulation method and convergence result} \label{sec:MainResult}
Our aim is to generate discrete samples $(X_{t_i}(y_k),\,i\leq N,k\leq M)$ of the process defined via \eqref{eq:SPDEintro} at the equidistant points
$$y_k=\frac{k}{M},\;k=0,\ldots,M,\qquad t_i = \frac{iT}{N},\;i=0,\ldots N,$$ 
where all of the numbers $N,M \in \N$ and $T>0$ are allowed to tend to infinity, in general. For the temporal and spatial mesh sizes we write
$$\Delta := t_{i+1}-t_{i}= \frac{T}{N},\qquad \delta:=y_{k+1}-y_k=\frac{1}{M}.$$\\
From representation \eqref{eq:X.yk} it is clear that sampling from $X$
at the grid points $(t_i,y_k)$ is equivalent to sampling from the processes $U_m,\,m\leq M-1$ at times $t_0,\ldots ,t_N$. Further, any coefficient process $u_\ell$ may be simulated exactly using its AR(1)-structure, namely
\begin{equation} \label{eq:AR1}
u_\ell(0)=\langle\xi,\e_\ell \rangle_\vt,\qquad u_\ell(t_{i+1})= \e^{-\lambda_\ell \Delta}u_\ell(t_{i})+\sigma \sqrt{\frac{1-\e^{-2\lambda_\ell \Delta}}{2\lambda_\ell}}N_{i}^\ell,\quad i\in \N,
\end{equation}
where $(N_i^\ell )$ are independent standard normal random variables. \\

To derive the simulation method let us first assume that $X_0=0$. In this case, the coefficient processes $u^0_k$ are centered Gaussian with covariance function
$$\Cov(u^0_\ell(t_i), u^0_\ell(t_j))= \frac{\sigma^2}{2\lambda_\ell}\e^{-\lambda_\ell|i-j|\Delta}\left(1-\e^{-2\lambda_\ell \min(i,j)\Delta} \right),\qquad 0\leq i,j\leq N.$$  
Thus, when $\lambda_\ell\eqsim \ell^2$ is large compared to $1/\Delta$, the random variables $(u^0_\ell(t_i),1\leq i\leq N)$ effectively behave like  iid Gaussian random variables with variance 
$$\V(u^0_\ell(t_i)) \approx \frac{\sigma^2}{2\lambda_\ell},\quad 1\leq i\leq N,$$
due to the exponential factor $\e^{-\lambda_\ell|i-j|\Delta}$ in the covariance.
Now, in order to define the approximation of the processes $U_m^0$,
choose $L= L_{M,N}\in \N$ and replace all coefficient processes $(u_\ell(t_i),\,1\leq i\leq N)$ with $\ell \geq LM$ by a vector of independent normal random variables with variance $\sigma^2/(2\lambda_\ell)$. Hereby, counting in multiples of $M$ is convenient due to the particular form of the index sets $\mathcal I_m = \mathcal I_m^+ \cup \mathcal I_m^- $. Since the normal distribution is stable with respect to summation, for each $m<M$  it is sufficient to generate one set $(R_m^{0,L}(i),\,1\leq i\leq N)$ of independent random variables with $R_m^{0,L}(i)\sim \mathcal N(0,s_m^2)$, where 
\begin{equation} \label{eq:def.sm2}
s_m^2= \sum_{\ell \in \mathcal I_m,\,\ell \geq  LM} \frac{\sigma^2}{2\lambda_\ell}
\end{equation}
and the resulting approximation is defined by 
\begin{equation*} 
U_m^{0,L}(0)=0,\qquad U_m^{0,L}(t_i)=\sum_{\ell \in \mathcal I_m,\ell < LM} u^0_{\ell}(t_i)+ R_m^{0,L}(i),\quad 1\leq i\leq  N.
\end{equation*}

Similarly, for the stationary solution $X^{\mathrm{st}}$, the coefficient processes $u_\ell^{\mathrm{st}}$ are centered Gaussian with covariance function
$$\Cov(u^{\mathrm{st}}_\ell(t_i), u^{\mathrm{st}}_\ell(t_j))= \frac{\sigma^2}{2\lambda_\ell}\e^{-\lambda_\ell|i-j|\Delta},\qquad 0\leq i,j\leq N.$$  Consequently, for iid random variables $(R_m^{\mathrm{st},L}(i),\,0\leq i\leq N)$ with $R_m^{\mathrm{st},L}(i)\sim \mathcal N(0,s_m^2)$ we define the approximation 
\begin{equation*}
U_m^{\mathrm{st},L}(t_i)=\sum_{\ell \in \mathcal I_m,\ell < LM} u^{\mathrm{st}}_{\ell}(t_i)+ R_m^{\mathrm{st},L}(i),\quad 0\leq i\leq  N.
\end{equation*}

In order to generate samples based on the replacement method it is necessary to calculate the variances $s_m^2$. Hereby, approximating the infinite series \eqref{eq:def.sm2} can be avoided thanks to the closed form expression provided by the following lemma.
\begin{lem} \label{lem:sm2}
Let $\Gamma = \frac{\vt_1^2}{4\vt_2^2}-\frac{\vt_0}{\vt_2}$, $\Gamma_0 = \sqrt{|\Gamma|}$ and define $\Sigma \in \R^{(M+1)\times (M+1)}$ via $ \Sigma_{kl}= \rho(y_k,y_l)$, where $\rho:[0,1]^2\to\R$ is the symmetric function given by   
$$
\rho(x,y)=\frac{\sigma^2}{2\vartheta_2} \cdot 
\begin{cases} 
\frac{\sin(\Gamma_0 (1-y))\sin(\Gamma_0 x)}{\Gamma_0\sin(\Gamma_0)}, & \Gamma <0,\\
x(1-y), & \Gamma=0,\\
\frac{\sinh(\Gamma_0(1-y))\sinh(\Gamma_0 x)}{\Gamma_0 \sinh(\Gamma_0)}, & \Gamma >0,
\end{cases}
\qquad
\text{for } x\leq y.
$$
Further, let $b_m= \sqrt 2 (\sin(\pi m y_0),\ldots,\sin(\pi m y_M))^\top\in \R^{M+1}$. The variance $s_m^2$ defined by \eqref{eq:def.sm2} satisfies
\begin{equation} \label{eq:formula.sm2}
s_m^2 =\frac{1}{M^2}b_m^\top\Sigma b_m - \sum_{\ell \in \mathcal I_m,\,\ell < LM }\frac{\sigma^2}{2\lambda_\ell}.
\end{equation}
  
\end{lem}

Our simulation method is summarized in the following algorithm:
{
\begin{alg}[Replacement method] \label{alg:replacement}
Choose $L\in \N.$\\[3mm]
For $1\leq m < M$ do the following:
	\begin{enumerate}[(1)]
	\item For $\ell \in \mathcal I_m\cap (0,LM)$ simulate $(u_\ell(t_i),\,0\leq i\leq N)$ according to $\eqref{eq:AR1}$.
	\item Compute $s_m^2$ according to \eqref{eq:formula.sm2} and generate  $R^L_m(0),\ldots,R^L_m(N)\sim \mathcal N(0,s_m^2)$ independently. For the zero initial condition replace $R_m^L(0)$ by $0$.
	\item Compute $$U^L_m(t_i)=\sum_{\ell \in \mathcal I_m,\ell < LM} u_{\ell}(t_i)+ R_m^L(i),\quad 0\leq i\leq  N.$$ 
	\end{enumerate}
\textbf{Output:}  $X^{L}_{t_i}(y_k)= \sum_{m=1}^{M-1}U^L_m(t_i)e_m(y_k)$ for $0\leq k\leq M$ and $0\leq i \leq N$.
\end{alg}
Assuming a finite set of observations, \citet{DavieGaines2001} proposed the replacement method with $L=1$, while omitting a detailed analysis. 
The following theorem theoretically justifies their approach and, allowing for $M,N\to \infty$, it provides a condition on $L$ for the validity of the approximation in total variation distance.   
}

\begin{thm} \label{thm:TV}
Let $\mathcal X$ be the  observation vector $\mathcal X=(X_{t_i}(y_k),\,i\leq N,k\leq M)$ either with zero or with stationary initial condition and let $\mathcal X^L$ be its approximation computed via Algorithm \ref{alg:replacement}. 
\begin{enumerate}[(i)]
\item There exist constants $c,C>0$ only depending on the parameters $(\sigma^2,\vt)$ such that 
$$\mathrm{TV}(\mathcal X,\mathcal X^L)\leq C\, \sqrt{MN} \e^{-cL^2M^2 \Delta}.$$
\item Assume $T\Delta^q \to 0$ for some $q>0$.
If there exists  $\alpha>1/2$ such that $LM \Delta^{\alpha} \to \infty$, then
$\mathrm{TV} (\mathcal X,\mathcal X ^L)\to 0$. In particular, if $T= \mathrm{const.}$ and $M / N^{\alpha} \to  \infty$ for some $\alpha>1/2$, then  $\mathrm{TV} (\mathcal X,\mathcal X^1)\to 0$.
\end{enumerate}
\end{thm}

A negligible total variation distance is exactly what is required for {statistical simulations} since functionals based on true  and approximate data share the same limiting distribution: let $(X_{n,k})$ and $(Y_{n,k})$ be triangular arrays of the same size and assume that $\phi_n(X_{n,\bullet})$ has a weak limit $Z$ for some sequence of functionals $\phi_n$. Then, if $\mathrm{TV}(X_{n,\bullet},Y_{n,\bullet}) \to 0$, the sequence $\phi_n(Y_{n,\bullet})$ also converges to $Z$ weakly. In fact, if $\mu_n$ is a dominating measure for $\P^{X_{n,\bullet}}$ and $\P^{Y_{n,\bullet}}$ with corresponding Radon-Nikodym derivatives $f_{X_{n,\bullet}}$ and $f_{Y_{n,\bullet}}$, then 
\begin{align*}
\big|\E(\e^{it\phi_n(X_{n,\bullet})})-\E(\e^{it\phi_n(Y_{n,\bullet})})\big|&=\Big|\int \e^{it\phi_n(z)}(f_{X_{n,\bullet}}(z)-f_{Y_{n,\bullet}}(z))\,\mu(dz) \Big| \\
&\leq  \Vert f_{X_{n,\bullet}}-f_{Y_{n,\bullet}}\Vert_{L^1(\mu)}= 2 \mathrm{TV}(X_{n,\bullet},Y_{n,\bullet}).
\end{align*}
Thus, the limiting characteristic functions coincide.\\
{Another aspect worth noting is that there is no statistical test that can consistently distinguish between two models whose total variation distance tends to zero: in such a case, the maximum of type one and type two error of any test for the true model is asymptotically bounded from below by $1/2$, see e.g.~\cite[Theorem 2.2]{Tsybakov10}.}\\


\section{Simulations} \label{sec:Simulations}

\begin{figure} 
\begin{subfigure}[b]{.5\textwidth}
\centering
\includegraphics[scale=.4]{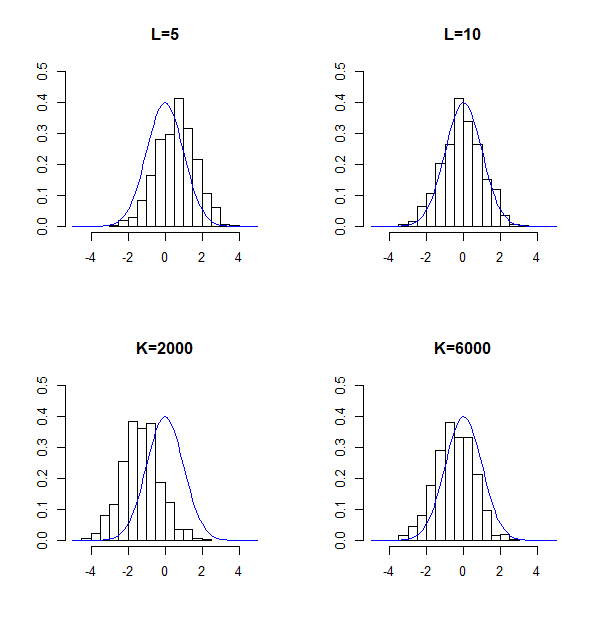}
\caption{temporal quadratic variation for\\ $N=5,000,\,M=10$ }
\label{fig:TI_comparison}
\end{subfigure}
\begin{subfigure}[b]{.5\textwidth}
\centering
\includegraphics[scale=.4]{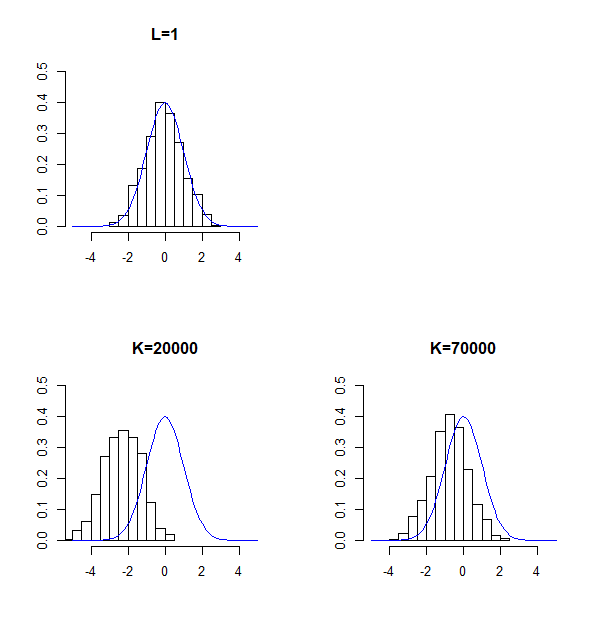}
\caption{spatial quadratic variation for\\ $N=100,\,M=1,000$  }
\label{fig:SI_comparison}
\end{subfigure}
\caption{Histograms based on 500 Monte Carlo iterations  for normalized quadratic variations based on the replacement (top) and truncation method (bottom). The solid line corresponds to the standard normal density function.}
\end{figure}

In order to test the performance of the replacement method and compare it to truncation of the Fourier series, we compute rescaled realized temporal and spatial quadratic variations, namely
\begin{align*}
V_{\mathrm t} &=\frac{1}{MN\sqrt \Delta}\sum_{i=0}^{N-1}\sum_{k=0}^{M-1}\e^{\kappa y_k}(X_{t_{i+1}}(y_{k})-X_{t_i}(y_{k}))^2,\\
V_{\mathrm{sp}} &=\frac{1}{MN\delta}\sum_{i=0}^{N-1}\sum_{k=0}^{M-1}\e^{\kappa y_k}(X_{t_{i}}(y_{k+1})-X_{t_i}(y_{k}))^2
\end{align*} 
based on both methods on the finite time horizon $T =1$. The outcomes are compared with the following theoretical results:
As shown in \cite{Bibinger18}, the temporal quadratic variation satisfies for any finite $M$
\begin{align*}
\sqrt{MN} \Big(V_{\mathrm t}-\frac{\sigma^2}{\sqrt{\pi\vt_2}}\Big) \overset{\mathcal D}{\longrightarrow} \mathcal N\Big(0,\frac{B\sigma^4}{\pi \vt_2}\Big),\qquad N\to \infty,
\end{align*}
where $B=2+ \sum_{j=1}^\infty (2\sqrt j - \sqrt{j+1} - \sqrt{j-1})^2$. In fact, the central limit theorem also holds for both $M,N\to \infty$ when considering equidistant spatial locations $(\tilde y_k,0\leq k\leq M)\subset [b,1-b]$  with  $M = o(N^\rho)$ for for some $b>0$ and $\rho<1/2$. Concerning the spatial quadratic variation, it is shown in \cite{HildebrandtTrabs2019} that if $N=o(M)$, then
\begin{align*}
\sqrt{MN} \Big(V_{\mathrm{sp}}-\frac{\sigma^2}{2\vt_2}\Big) \overset{\mathcal D}{\longrightarrow} \mathcal N\Big(0,\frac{\sigma^4}{2 \vt_2^2}\Big),\qquad M,N\to \infty.
\end{align*}
This central limit theorem is also valid when $N$ remains finite.
For the simulations we have set the parameters to the values $\sigma^2=0.1,\,\vt_2 = 0.5,\,\vt_1= -0.4,\,\vt_0 = 0.3$ and have considered the corresponding stationary initial condition. Each of the plots in Figures \ref{fig:TI_comparison} and \ref{fig:SI_comparison} shows a histogram of the centered and normalized (with respect to theoretical asymptotic means and variances) quadratic variations based on 500 Monte Carlo iterations. The solid line corresponds to the standard normal density function.\\
For the temporal quadratic variation (Figure \ref{fig:TI_comparison}) we have considered $M = 10$ spatial and $N=5,000$ temporal observations. It can be seen that the values provided by the replacement method with $L=10$ (corresponding to $L M = 100$ simulated Ornstein-Uhlenbeck processes) is already in good accordance with the theoretical limit. Note that $LM\sqrt \Delta\approx 3.2$ is far from infinity, so the method works better than predicted by Theorem \ref{thm:TV}. The truncation method, on the other hand, requires simulation of more than $6,000$ coefficient processes in order to produce accurate results and prevent a severe bias in the simulated values.  \\
Examining the results for the spatial quadratic variation (Figure \ref{fig:SI_comparison}), this effect becomes even more apparent.
Here, we considered $M = 1,000$ spatial and $N=100$ temporal observations. Consequently, $M\sqrt \Delta = 100$ and Theorem \ref{thm:TV} suggests that $L=1$ (i.e.~$LM=1,000$ simulated coefficient processes) is sufficient for the replacement method. Figure \ref{fig:SI_comparison} confirms this prediction. On the other hand, even with $K=70,000$ coefficient processes, the simulated values based on the truncation method still suffer from a severe bias.\\

In fact, the bias in the central limit theorems introduced by truncation can be explained analytically: A simple calculation shows that for the normalized temporal quadratic variation, the bias is of order $\sqrt{MN} \frac{1}{\sqrt \Delta}\sum_{\ell \geq K}\frac{1}{\lambda_\ell} \eqsim\frac{\sqrt{MN}}{K\sqrt \Delta}$ and in our simulation for the temporal quadratic variation we have $\frac{\sqrt{MN}}{\sqrt \Delta}\approx 16,000$.
Similarly, the bias for the spatial quadratic variation is of order $\sqrt{MN}\frac{1}{\delta}\sum_{\ell \geq K} \frac{1}{\lambda_\ell}\eqsim\frac{\sqrt{MN}}{K\delta}$, in our simulation we have $\frac{\sqrt{MN}}{\delta}\approx 316,000$.

\section{Proofs}\label{sec:Proofs}
First, we prove the closed form expression for the variances $s_m^2$:
\begin{proof}[Proof of Lemma \ref{lem:sm2}]
It follows from \cite[Proposition 2.1]{HildebrandtTrabs2019} that $\Sigma$ is the covariance matrix of the vector $\tilde X_0^{\mathrm{st}}(y_\cdot)=(\e^{\kappa y_1/2}X_0^{\mathrm{st}}(y_0),\ldots ,\e^{\kappa y_M/2}X^{\mathrm{st}}_0(y_M))^\top$. Therefore, the claimed formula follows from
$$\sum_{\ell \in \mathcal I_m} \frac{\sigma^2}{2\lambda_\ell} = \V \left(\langle  X_0^{\mathrm{st}}(\cdot),{e}_m \rangle_M \right) =\frac{1}{M^2}\V \left(b_m^\top \tilde X_0^{\mathrm{st}}(y_\cdot)\right) = \frac{1}{M^2}{b_m^\top}\Sigma b_m, $$ 
where the exponential factors cancel in the second step.
\end{proof}

Next, we prove our main result:
\begin{proof}[Proof of Theorem \ref{thm:TV}]
We make use of the result by \citet{Devroye19} that 
\begin{equation} \label{eq:TVbound}
\mathrm{TV}\big( \mathcal N(0,A), \mathcal N(0,B) \big) \leq \frac{3}{2}\Vert A^{-1}(B-A) \Vert_F
\end{equation}
holds for  positive definite matrices $A$ and $B$ of the same size.\\
First, we treat the case of a stationary initial condition and suppress the superscript $\mathrm{st}$ for the sake of convenience.
Since $\mathrm{TV}(f(X),f(Y))\leq \mathrm{TV}(X,Y)$ holds for any random vectors $X$ and $Y$ and any measurable function $f$, the problem can be reduced to bounding the total variation distance of $(U_m(t_i),\,i\leq N,m\leq M-1)$ from its approximation. Furthermore, since both $U_m$ and $U_m^L$ are made up of independent summands, it is sufficient to consider the parts of the sums in which the two differ. To that aim define $\mathcal R^L = (R^L_m(i),\,i\leq N,m\leq M-1)$ and $\mathcal V^L = (V^L_m(t_i),\,i\leq N,m\leq M-1)$, where $V^L_m(t) = \sum_{\ell \in \mathcal I_m,\,\ell \geq LM} u_{\ell}(t)$.
Let $\Xi_m$ be the covariance matrix of $(V^L_m(t_i),\,i\leq N)$ and $\Xi_m ^{\indi}$ be the covariance matrix of $(R^L_m(t_i),\,i\leq N)$ as well as $\Xi = \mathrm{diag}(\Xi_1,\ldots, \Xi_{M-1})$, $\Xi^{\indi} = \mathrm{diag}(\Xi_1^{\indi},\ldots, \Xi_{M-1}^{\indi})$. 
Since  $\mathcal V^L$ and $\mathcal R^L$ are centered Gaussian random vectors with covariance matrices $\Xi$ and $\Xi^{\indi}$, respectively, we can use \eqref{eq:TVbound} and the block structure to bound
\begin{align} \label{eq:TV_bound1}
\mathrm{TV} (\mathcal V^L,\mathcal R^L)^2 \leq \frac{9}{4} \Vert (\Xi^{\indi})^{-1}(\Xi-\Xi^{\indi}) \Vert_F^2 = \frac{9}{4} \sum_{m=1}^{M-1} \Vert (\Xi_m^{\indi})^{-1}(\Xi_m-\Xi^{\indi}_m) \Vert_F^2. 
\end{align}
We treat each term separately. Note that $\Xi_m^{\indi}$ is a diagonal matrix with the same diagonal elements as $\Xi_m$. 
Therefore, by the monotonicity of the exponential function,
\begin{align*}
\Vert (\Xi_m^{\indi})^{-1}(\Xi_m-\Xi_m^{\indi}) \Vert_F^2 
&= \frac{1}{s_m^4} \sum_{i\neq j} \left( \sigma^2\sum_{\ell \in \mathcal I_m,\,\ell \geq LM}\frac{\e^{-\lambda_\ell |i-j|\Delta}}{2\lambda_\ell}\right)^2\\
&\leq \frac{1}{s_m^4} \sum_{i\neq j} \left( \sum_{\ell \in \mathcal I_m,\,\ell \geq LM}\frac{\sigma^2}{2\lambda_\ell}\right)^2\e^{-2\lambda_{LM} |i-j|\Delta}\\
&=  \sum_{i\neq j}\e^{-2\lambda_{LM} |i-j|\Delta}.
\end{align*}
Using $\sum_{i=1}^\infty q^i= \frac{q}{1-q}$ for $|q|<1$, we can proceed to
\begin{align*}
\sum_{i\neq j} \e^{-2\lambda_{LM} |i-j|\Delta} \leq 2N \sum_{i=1}^\infty \e^{-2\lambda_{LM} i\Delta}
= 2N  \frac{\e^{-2\lambda_{LM} \Delta}}{1-\e^{-2\lambda_{LM} \Delta}}
\lesssim N  \e^{-2\lambda_{LM} \Delta},
\end{align*}
where the last step follows from the fact that  $L^2M^2\Delta \geq (LM\Delta^\alpha)^2 \to \infty$.
Now, letting $c>0$ be such that $c \ell^2 \leq \lambda_\ell$ for all $\ell \in \N$, we get the overall bound on the total variation distance claimed in $(i)$, namely
\begin{align*} 
\mathrm{TV} (\mathcal X,\mathcal X^L)^2\leq\mathrm{TV} (\mathcal V^L,\mathcal R^L)^2\lesssim MN \e^{-2\lambda_{LM}\Delta}\leq MN \e^{-2c L^2 M^2 \Delta}.
\end{align*}
{To prove $(ii)$, choose $r>0$ such that $\frac{r+q+1}{2r-1}\leq \alpha$. Then, using $(i)$ and $\exp(-x)\lesssim x^{-r}, x>0,$ for any $r>0$, we find $$\mathrm{TV} (\mathcal X,\mathcal X^L)^2\lesssim MN\e^{-2c L^2 M^2 \Delta}\lesssim \frac{MT}{(LM)^{2r}\Delta^{r+1}}=\frac{T\Delta^q}{L} \left(\frac{1}{LM \Delta ^{\frac{r+q+1}{2r-1}}} \right)^{2r-1} \to 0,$$} 
finishing the proof for the stationary case.\\

The case $X_0=0$ works similarly: Again, let $\Xi_m$ be the covariance matrix of $(V^L_m(t_i),\,i\leq N,m\leq M-1)$ and $\Xi_m^{\indi}$ be the covariance matrix of $(R^L_m(i),\,i\leq N,m\leq M-1)$ (without the initial deterministic value). Clearly, bound \eqref{eq:TV_bound1} remains valid and
\begin{align*}
\Vert (\Xi_m^{\indi})^{-1}(\Xi_m-\Xi_m^{\indi}) \Vert_F^2 
&=\frac{1}{s_m^4} \sum_{i,j=1}^N \left( \sigma^2\sum_{\ell \in \mathcal I_m,\,\ell \geq LM}\frac{\e^{-\lambda_\ell |i-j|\Delta}}{2\lambda_\ell}\left(1-\delta_{ij}-\e^{-2\lambda_\ell(i\wedge j)\Delta} \right)\right)^2\\
&= \frac{1}{s_m^4} \sum_{i\neq j} \left( \sigma^2\sum_{\ell \in \mathcal I_m,\,\ell \geq LM}\frac{\e^{-\lambda_\ell |i-j|\Delta}}{2\lambda_\ell}\left(1-\e^{-2\lambda_\ell(i\wedge j)\Delta} \right)\right)^2 \\
&\qquad+ \frac{1}{s_m^4}\sum_{i=1}^N \left( \sigma^2\sum_{\ell \in \mathcal I_m,\,\ell \geq LM}\frac{1}{2\lambda_\ell}\e^{-2\lambda_\ell i\Delta} \right)^2\\
&\leq \frac{1}{s_m^4} \sum_{i\neq j} \left( \sigma^2\sum_{\ell \in \mathcal I_m,\,\ell \geq LM}\frac{\e^{-\lambda_\ell |i-j|\Delta}}{2\lambda_\ell}\right)^2 \\
&\qquad+ \frac{1}{s_m^4}\sum_{i=1}^N \left( \sigma^2\sum_{\ell \in \mathcal I_m,\,\ell \geq LM}\frac{1}{2\lambda_\ell}\e^{-2\lambda_\ell i\Delta} \right)^2\\
&\leq \frac{2}{s_m^4} \sum_{i\neq j} \left( \sigma^2\sum_{\ell \in \mathcal I_m,\,\ell \geq LM}\frac{\e^{-\lambda_\ell |i-j|\Delta}}{2\lambda_\ell}\right)^2,
\end{align*}
from which the result follows as in the stationary case.
\end{proof}

\section*{Acknowledgement} I would like to thank my Ph.D.~advisor, Mathias Trabs, for the careful reading of this manuscript and his useful suggestions.

\bibliography{refs}
\bibliographystyle{apalike}
\end{document}